\documentclass[a4paper, 11pt]{article}

\usepackage{amssymb}
\usepackage{amsfonts}
\usepackage{amsmath}
\usepackage{amsthm}
\usepackage{graphicx}

\usepackage{fancyhdr}

\begin{document}

\title{Compact Clifford--Klein Forms -- Geometry, Topology and Dynamics}
\author{David Constantine
\footnote{\textsc{Department of Mathematics and Computer Science, Wesleyan University, Middletown, CT 06459 U.S.A.}                  
\textit{email:} \texttt{dconstantine@wesleyan.edu}}}
\date{\today}
\maketitle

%%%%%%%define theorems
%\theoremstyle{plain} \newtheorem*{thmA}{Theorem A}
%\theoremstyle{plain} \newtheorem*{thmB}{Theorem B}
\theoremstyle{plain} \newtheorem{thm}{Theorem}[section]
\theoremstyle{plain} \newtheorem{lemma}[thm]{Lemma}
\theoremstyle{plain} \newtheorem{prop}[thm]{Proposition}
\theoremstyle{plain} \newtheorem{cor}[thm]{Corollary}	
\theoremstyle{definition} \newtheorem{defn}[thm]{Definition}	
\theoremstyle{remark} \newtheorem*{note}{Notation}	
\theoremstyle{remark} \newtheorem{rmk}[thm]{Remark}
\theoremstyle{remark} \newtheorem{obs}[thm]{Observation}
\theoremstyle{remark} \newtheorem{ques}[thm]{Question}
\theoremstyle{plain} \newtheorem{conj}[thm]{Conjecture}

\newcommand{\dq}{H\backslash G/\Gamma}
\newcommand{\SL}[1]{\textrm{SL}_{#1}(\mathbb{R})}
\newcommand{\transversal}{\ensuremath{\bar\pitchfork}}

\maketitle
%ABSTRACT%%%%%%%%%
\begin{abstract}
We survey results on compact Clifford--Klein forms of homogeneous spaces, with a focus on recent contributions and organized around approaches via topology, geometry and dynamics.  In addition, we survey results on moduli spaces of compact forms.
\end{abstract}

\tableofcontents

%
%%
%%%
%%%%
%%%%%
%%%%%%
%%%%%%%
%%%%%%%%%%%%%%%%%%%%%%%%%%%%%%%%%%
\section{Introduction}

\subsection{The basic questions, and a conjecture}

This survey will be devoted to two questions.  The first, the existence question, has been studied fairly extensively.  The second, the deformation question, has so far received less attention, though there have been several recent developments.

\begin{ques}[The existence question for compact Clifford--Klein forms]\label{exist_quest}
Let $H\backslash G$ be a homogeneous space of a noncompact Lie group $G$.  Does there exist a discrete subgroup $\Gamma$ in $G$ such that $\dq$ is a compact manifold locally modeled on $H\backslash G$?
\end{ques}

\noindent  A compact manifold $\dq$ is generally called a compact Clifford--Klein form.  The alternate terminology `$H\backslash G$ has a tessellation' has been suggested by the series of papers \cite{Oh_Witte_new, Oh_Witte_compact, Witte_Iozzi}.  This terminology may be preferable in that it avoids the confusion with compact forms of Lie groups, but it is not standard.

The ultimate goal is to understand all pairs $(G,H)$ which have compact Clifford--Klein forms.  There is a Goldilocks tension at play in looking for $\Gamma$: it must not be too big, so that it can act freely and properly discontinuously on $H\backslash G$; it must not be too small or else $\dq$ will not be compact.  In addition, $\Gamma$ must be `well-positioned' in $G$ relative to $H$, as we will see most clearly in the approaches to the problem in Section \ref{sec_cartan_proj} below.

For all of the results we will outline, some sort of restriction must be placed on $G$ and $H$.  It is very common to assume $G$ semisimple, or even simple, and that $H$ is reductive in $G$, i.e. its adjoint action on the Lie algebra of $G$ is reducible.  These are the main cases of geometric interest and will be the focus of this survey.  As exceptions to this rule, however, I would note: the flat Minkowski spacetime $\mathbb{R}^{n,1}=\textrm{O}(n,1)\backslash(\textrm{O}(n,1)\ltimes \mathbb{R}^{n+1})$ and the affine space $\mathbb{R}^n=\textrm{GL}_n(\mathbb{R}) \backslash (\textrm{GL}_n(\mathbb{R}) \ltimes \mathbb{R}^n)$, especially as studied in relation to the Auslander Conjecture (see \cite{gold_et_al} and \cite{aus_survey} for surveys); and the `tangential symmetric spaces,' studied in \cite{ksurvey} by Kobayashi and Yoshino.

The study of such forms of Riemannian homogeneous spaces goes back to the work of Borel on discrete subgroups of Lie groups (see \cite{borel-forms}).  As we will soon see, however, the main focus of Question \ref{exist_quest} is now pseudo-Riemannian homogeneous spaces, i.e. those for which $H$ is noncompact.  The first work on such spaces was that of Calabi and Markus \cite{calabi-markus}.  Since the late '80's, the driving force behind considering the problem for reductive, non-Riemannian homogeneous spaces -- our specific concern here --  has been Toshiyuki Kobayashi (see the many references below), and in subsequent years a wide variety of mathematicians have contributed to the field.

Let us note two related problems. First, one can ask which homogeneous spaces have finite-volume Clifford--Klein forms.  A few of the results below are general enough to deal with this situation, but to the author's knowledge nothing specific to the finite-volume case has been done.  It would be quite surprising to find a homogeneous space possessing a finite-volume form but no compact one.

Second, one can also ask whether a compact $(G,X)$-manifold where $X=H\backslash G$ exists.  Such a structure consists of the following: a \emph{developing map} $dev:\tilde M \to X$ and a \emph{holonomy homomorphism} $hol:\pi_1(M)\to G$ such that the developing map is a local diffeomorphism and equivariant with respect to the holonomy map.  $(G,X)$-structures will be used in Section \ref{sec_goldman} below, and are treated in more detail there.  A compact $(G,X)$-manifold need not be of the form $\dq$.  The author knows of two results flexible enough to treat this more general existence question (\cite{bl} and \cite{lmz}, Sections \ref{ben_lab} and \ref{sec_lmz} below), but again, most cases are open.

\vspace{.5cm}

Naturally, the second question we survey is the uniqueness counterpart of the first:

\begin{ques}[The deformation question for compact forms]\label{def_quest}
For homogeneous spaces possessing a compact form, can one describe the space of all compact forms?  Specifically, can compact forms be deformed locally?
\end{ques}

Work on this question is, generally speaking, more recent, and includes some very recent results which I will survey in Section \ref{sec_deformations}.

\vspace{.5cm}

There is one central conjecture in this field, having to do with the following algebraic construction of compact forms.  Let us suppose that $L$ and $H$ are Lie subgroups of $G$ such that:

\begin{itemize}
	\item $HL=G$
	\item $H\cap L$ is compact
	\item $\Gamma$ is a uniform lattice in $L$
\end{itemize}

\noindent Then a compact form of $H\backslash G$ may be formed by taking 

\[H\backslash G/\Gamma \cong (H\cap L) \backslash L/\Gamma.\] 

\noindent The requirement above that $HL=G$ can be loosened -- we need only that $HL$ is cocompact in $G$.  These algebraically-constructed compact forms will be called \emph{standard forms}.  Kobayashi conjectures the following:

\begin{conj}[Kobayashi,  \cite{ksurvey} Conj. 3.3.10]\label{kob_conj}
If $H\backslash G$ possesses a compact form, for $G$ and $H$ reductive, it possesses a standard form.
\end{conj}

\noindent The reader should note that Kobayashi's conjecture is not simply `all compact forms are standard.'  We will see in Section \ref{sec_deformations} that the answer to the deformation question is sometimes yes, and that not all the forms are standard.  The first such results were due to Goldman, Ghys and Kobayashi (\cite{goldman_nonstd}, \cite{ghys2}, \cite{kobayashi_def}).  In addition, there are nonstandard forms which are not deformations of standard ones, first obtained by Salein (\cite{salein_exotic}, see Section \ref{sec_salein_kassel}).

\begin{rmk}\label{Oniscik_rmk}
I would like to note here a result that bears on the construction mentioned above.  In \cite{compact_forms} Lemma 9.2.2, we prove that when $G$ is simple and $H$ is contained in a maximal proper parabolic subgroup of $G$, then the condition $HL=G$ forces $L$ to be semisimple.  Assuming $H$ reductive, one can then apply the work of Oni\v{s}\v{c}ik in \cite{oniscik}, where all decompositions of simple $G$ as products of reductive subgroups are classified.\footnote{Oni\v{s}\v{c}ik's work assumes that $H$ is reductive but not that $H$ is contained in a proper parabolic subgroup.}  Most examples in this classification do not satisfy $H\cap L$ compact; the remainder of Oni\v{s}\v{c}ik's list provides a useful catalogue of spaces potentially having compact forms.
\end{rmk}

We will survey below some of the results that give evidence for Conjecture \ref{kob_conj}.  However, as the reader will gather from the wide variety of approaches to the problem (each of which has its own particular realm of usefulness), there is no unified approach at the moment.

%
%%
%%%
%%%%
%%%%%
%%%%%%
%%%%%%%
%%%%%%%%%%%%%%%%%%%%

\subsection{Motivation and special examples}

We motivate the existence question with a special case.  Let $H$ be a compact subgroup of $G$.  Then it is easy to see that finding a suitable $\Gamma$ reduces to finding a uniform lattice in $G$.  For $G$ semisimple, Borel found such lattices (see \cite{borel-forms}).  These examples include the Riemannian symmetric spaces, spaces of great geometric interest.

As we allow $H$ noncompact, we can still form symmetric spaces by taking $H$ to be an open subgroup of the points of $G$ fixed by some involution.  In general, however, these will be only \emph{pseudo-Riemannian} symmetric spaces (see Prop \ref{signature_prop}).  The main motivation for the existence question is to understand these symmetric spaces as we do the Riemannian ones, although they are, generally speaking, much harder to study.  The main difficulty is that $\dq$, with both $H$ and $\Gamma$ noncompact, usually does not inherit the main structures present on $G$; a good example is a right-invariant Riemannian metric.  Hence the need for the many new techniques developed below.

To motivate the deformation question, consider the special case $G=\SL 2$, $H=\textrm{SO(2)}$.  Here the deformation spaces of compact forms (up to conjugation of $\Gamma$ in $G$) are Teichm\"uller spaces of marked hyperbolic structures on closed, genus $g\geq 2$ surfaces.  Some of the deformation results we will examine below are actually on a similar problem, examining again constant curvature -1 structures, but this time on three-dimensional Lorentzian manifolds.

\vspace{.5cm}

A few special examples are worth mentioning.  The first example one might consider -- simply because of the straight-forward structure of the groups involved -- is $\SL k \backslash \SL n$, where $\SL k$ is embedded in the upper left-hand corner of $\SL n$.  This can be generalized by embedding $\SL k$ via other representations.  Amazingly, the existence question for these spaces is not entirely solved, although all work so far gives a negative answer to that question.  A variety of results deal with small-dimensional versions of this problem; the main results for general dimension follow from the topological approaches discussed in Section \ref{sec_topology} and the dynamical approaches in Section \ref{sec_dynamics}.

The second example is geometrically motivated, and predates the example above.  Call a pseudo-Riemannian manifold $M$ with signature $(p,q)$ a $(p,q)$-\emph{space form} if it has constant sectional curvature $\kappa$.  We are familiar, of course, with the Riemannian $(p,0)$-space forms; $(p,1)$-space forms are the Lorentzian case (see, e.g., \cite{wolf_spaces}).  The first motivation for the existence problem was Calabi and Markus's study of $(3,1)$-space forms as possible models for space-time (\cite{calabi-markus}, see Sec. \ref{sec_calabi_markus}).  If we assume as well that $M$ is geodesically complete, then $M$ must be a quotient of one of the following spaces (see \cite{wolf_spaces}):

\begin{itemize}
	\item $\kappa < 0$: $\mathbb{H}^{p,q}:=\textrm{O}(p,q)\backslash \textrm{O}(p, q+1)$ for $q\geq 2$, or $\widetilde{\textrm{O}(p,1)}\backslash \widetilde{\textrm{O}(p,2)}$
	\item $\kappa = 0$: $\mathbb{R}^{p,q}:=\textrm{O}(p,q)\backslash \textrm{O}(p,q) \ltimes \mathbb{R}^{p+q}$
	\item $\kappa >0$: $\mathbb{S}^{p,q}:=\textrm{O}(p,q) \backslash \textrm{O}(p+1, q)$ for $p\geq 2$, or $\widetilde{\textrm{O}(1, q)}\backslash \widetilde{\textrm{O}(2, q)}$.
\end{itemize}

Kobayashi has the following conjecture about pseudo-Riemannian space forms:

\begin{conj}[Kobayashi, \cite{ko_01}, Conj 2.6]
There exists a compact space form with signature $(p,q)$ ($p, q \neq 0$) and sectional curvature $\kappa$ if and only if:
\begin{itemize}
	\item $\kappa< 0$ and $(p,q)$ belongs to 
		\begin{tabular}{|c|c|c|c|c|c|}
			\hline
			$p$ & 1 & 3 & 7 \\ 
			\hline
			$q$ & $2\mathbb{N}$ & $4\mathbb{N}$ & 8 \\
			\hline
		\end{tabular}
	\item $\kappa=0$ and $(p,q)$ is arbitrary
	\item $\kappa> 0$ and $(p,q)$ belongs to 
		\begin{tabular}{|c|c|c|c|c|c|}
			\hline
			$p$ & $2\mathbb{N}$ & $4\mathbb{N}$ & 8 \\ 
			\hline
			$q$ & 1 & 3 & 7 \\
			\hline
		\end{tabular}
\end{itemize}
\end{conj}

\noindent  The `if' part of this conjecture is proven (see \cite{ksurvey} and references therein); the `only if' part is open, though we will see several results below that address portions of it.  The $\kappa=0$ portion is trivial; the $\kappa<0$ and $\kappa>0$ portions are equivalent as $\mathbb{S}^{p,q} \simeq \mathbb{H}^{q,p}$.

As preparation for a few of the techniques below, let me record a definition of an important class of homogeneous spaces and a fact about their pseudo-Riemannian structure.  

\begin{defn}
The homogeneous space $H\backslash G$ is of reductive type if $G$ is reductive, $H$ is a closed subgroup with finitely many connected components and $H$ is stable under a Cartan involution\footnote{Recall that a Cartan involution of the Lie algebra $\mathfrak{g}$ is an involution $\theta$ such that $-\kappa(X, \theta Y)$ is positive definite as bilinear form in $X$ and $Y$, where $\kappa$ is the Killing form. It plays a central role in the theory of semisimple and reductive Lie groups.  $H$ is stable under this involution if its Lie algebra is stable.} of $G$.
\end{defn}

\noindent Note that $H\backslash G$ of reductive type implies that $H$ is reductive; the converse holds if $G$ is reductive and $H$ is semisimple or algebraic.

\begin{prop}[See, e.g., \cite{ksurvey} Prop 3.2.7]\label{signature_prop} Let $H\backslash G$ be a homogeneous space of reductive type.  Then $H\backslash G$ has a right-$G$-invariant pseudo-Riemannian metric of signature 
\[(d(G)-d(H), dim(G)-dim(H)-d(G)+d(H))\]
induced by the Killing form on $\mathfrak{g}$. Here $d(-)$ is the dimension of the Riemannian symmetric space associated to the Lie group.
\end{prop}

%
%%
%%%
%%%%
%%%%%%%%%%%%%%%%%%%%%%
\subsection{Organization and aims of the survey}

I would be very remiss not to mention that there are already several excellent surveys of the existence question for compact forms.  In particular, let me note the survey ``Quelques r\'esultats r\'ecents sur les espaces localement homog\`enes compacts'' \cite{labourie-survey} by Labourie, the survey ``Compact Clifford--Klein forms of symmetric spaces -- revisited'' \cite{ksurvey} by Kobayashi and Yoshino, and the introduction to Kassel's thesis \cite{fanny_thesis}.  Kobayashi and Yoshino's survey in particular is very extensive and for the results related to Kobayashi's large amount of work on the problem, no better survey could be desired.  

It is not my intention to replace any of these excellent papers, but rather to supplement them.  In particular I would like to point out the following features that I hope will make this survey a serviceable supplement:

\begin{itemize}
	\item We survey a few results (eg. \cite{compact_forms}, \cite{salein_exotic}, \cite{kassel_deformation}, \cite{g-w_anosov}) which have been published since the most recent survey \cite{ksurvey}, or were not included there.
	\item This survey gives the first complete presentation of the representation-theoretic approaches to the problem (see Section \ref{sec_rep_thry}). 
	\item We will treat the extensive work on the problem by Kobayashi and his collaborators lightly; \cite{ksurvey} is still the best resource on this work.
	\item We provide the first survey of material on the deformation question, much of which has been developed recently (Section \ref{sec_deformations}).
\end{itemize}

As this survey is being written as part of the proceedings for the conference Geometry, Topology and Dynamics in Negative Curvature (Bangalore, 2010), I have organized the presentation around the themes of the conference.  In Section \ref{sec_topology} we survey approaches to the existence question that can be broadly classed as topological.  In Section \ref{sec_geom} we survey geometric approaches and in Section \ref{sec_dynamics} we survey approaches primarily using dynamics.  The reader will find, of course, that many of our results straddle these categories.  Finally, in Section \ref{sec_deformations} we survey results on the deformation question.

%
%%
%%%
%%%%
%%%%%
%%%%%%%%%%%%%%%%%%%
\subsection{Thanks}

I would like to thank the organizers of the conference Geometry, Topology and Dynamics in Negative Curvature (Bangalore, 2010) for inviting me to speak at that excellent event, and for inviting me to contribute this survey to the proceedings.  I would especially like to thank Fanny Kassel for reading an earlier draft of this survey and providing numerous and very helpful comments and suggestions, and an unnamed referee for a very careful reading of an earlier draft and helpful suggestions too numerous to detail.  Remaining mistakes are my own.

%NOTATION%%%%%%%%%%
\section{Notation}
The following notation will be fixed throughout the survey for consistency.  Notation which only makes an appearance relating to one result will conform, as much as possible, to the authors' original notation, for ease of reference.

\begin{itemize}
	\item $G$ and $H$ will be (real) Lie groups; $\Gamma$ will be a discrete subgroup of $G$
	\item $\dq$ will be a compact Clifford--Klein form of the homogeneous space $H\backslash G$.
	\item $Z_G(H)$ will be the centralizer of $H$ in $G$; often $J<Z_G(H)$ will be a particular subgroup of this centralizer, usually a semisimple group.
	\item German letters ($\mathfrak{g}$, $\mathfrak{j}$, etc.) denote the Lie algebras associated to the corresponding Lie groups.
	\item For a Lie group $G$, let $d(G)$ denote the dimension of the (Riemannian) symmetric space associated to $G$. 
\end{itemize}

%
%%
%%%
%%%%
%%%%%
%%%%%%
%%TOPOLOGY%%%%%%%%%%
\section{Topology}\label{sec_topology}

In basic formulation, the compact forms question is a topological one, and the first approach to it is via the algebra and topology of Lie groups.  Some of what is surveyed in this section is not strictly topological, but there is a significant contrast with the approaches found in Section \ref{sec_geom}, where stronger geometric structures are used.  

The topological approach has a long history, with the most extensive contribution being provided by Toshiyuki Kobayashi.  In fact, no better survey of the topological approaches to this problem can be found than his survey \cite{ksurvey} with Yoshino.  Below I will present an abbreviated look at the basic features of the topological approaches; greater detail can be found in Kobayashi and Yoshino's survey.

%RANK CONSIDERATIONS%%%%%%
\subsection{Rank restrictions -- the Calabi--Markus phenomenon}\label{sec_calabi_markus}

Several of the main nonexistence results for compact forms utilize restrictions on the real ranks of the Lie groups $G$ and $H$.  In fact, the first nonexistence result is in this vein -- the so-called Calabi--Markus phenomenon.  In \cite{calabi-markus}, Calabi and Markus investigate what they call relativistic space forms, that is, complete Lorentz manifolds with constant curvature.  Of these there are three models, the Minkowski plane $\mathbb{R}^{n,1}$ with curvature zero, the de Sitter space $\textrm{dS}^n = \mathbb{S}^{n-1,1}$ with curvature $+1$, and the anti-de Sitter space $\textrm{AdS}^n = \mathbb{H}^{n-1,1}$ with curvature $-1$ (take $n\geq 2$ here).  The de Sitter space is formed by taking the standard Minkowski form on $\mathbb{R}^{n,1}$, namely  $\textrm{ds}^2 = -x_1^2+x_2^2+\cdots +x_{n+1}^2$ and restricting this form to the quadric hypersurface $-x_1^2+x_2^2+\cdots +x_{n+1}^2=1$.  The group $G=\textrm{SO}(1,n)$ preserves this hypersurface and the stabilizer of the point $(0,1,0, \ldots, 0)$ is $\textrm{SO}(1, n-1)$ so $\textrm{dS}^n = \textrm{SO}(1, n-1)\backslash \textrm{SO}(1,n)$.  Topologically, $\textrm{dS}^n$ is the product of the real line with a sphere, hence noncompact.

Calabi and Markus prove the following, of which the nonexistence of a compact (or finite-volume) form is a clear corollary.

\begin{thm}[Calabi--Markus, \cite{calabi-markus} Thm 1]  Any $\Gamma$ which acts properly discontinuously by isometries on $\emph{dS}^n = \emph{SO}(1, n-1)\backslash \emph{SO}(1,n)$ is finite.
\end{thm}

\begin{proof}[Idea of proof]
Consider the equatorial sphere in $\textrm{dS}^n$ described by setting $x_1=0$.  It is an easy exercise using the symmetry about the origin to show that the image of this sphere under any isometry intersects the original sphere.  As this sphere is compact, these intersections cannot occur for infinitely many elements of $\Gamma$ without violating proper discontinuity of the action.
\end{proof}

The general idea here -- that $\textrm{SO}(1,n-1)$ is `large enough' with respect to $\textrm{SO}(1,n)$ to leave no room for an infinite $\Gamma$ -- is codified more generally in the following criterion.

\begin{thm}[Calabi--Markus phenomenon, see \cite{calabi-markus}, \cite{wolf}, \cite{kobayashi-proper}]\label{cm}
If the $\mathbb{R}$-rank of $H$ is the same as the $\mathbb{R}$-rank of $G$, only a finite $\Gamma$ can act properly discontinuously on $H\backslash G$.  Thus $H\backslash G$ admits a compact form only if it is already compact.
\end{thm}

%
%%
%%%%%%%%%%%%%
\subsection{Generalizations of Calabi--Markus}\label{sec_cartan_proj}

The most comprehensive approach to this situation has been provided by Kobayashi and Benoist (see \cite{kobayashi-proper}, \cite{ko_necessary}, \cite{kob_criterion}, \cite{kobayashi_lecture}  and \cite{benoist_actions}, \cite{benoist}).  Benoist and Ko\-bay\-ashi formulate a very general criterion for proper action which suppresses even the group structures of $H$ and $\Gamma$ and treats them symmetrically.

\begin{defn}[\cite{kobayashi_lecture}, Defn 1.11.1; \cite{benoist}, \S 3.1]\label{kob_defn}
Let $A$ and $B$ be subsets of $G$.
\begin{itemize}
	\item The pair $(A, B)$ is \emph{transversal}, denoted $A \transversal B$, if $A\cap SBS$ is relatively compact for any compact $S\subset G$.\footnote{Benoist calls the pair $(A,B)$ `$G$-proper.'}
	\item The pair $(A,B)$ is \emph{similar}, denoted $A\sim B$, if there exists a compact subset $S$ in $G$ such that $A\subset SBS$ and $B\subset SAS$.
\end{itemize}
\end{defn}

\noindent Note that $\Gamma$ acts properly on $H\backslash G$ if and only if $\Gamma \transversal H$, that transversality is preserved by similarity, and that if $H\sim G$, then $H\transversal \Gamma$ only if $\Gamma$ is compact.  For $G\cong \mathbb{R}^n$, and $A$ and $B$ closed cones in $G$, the definitions are particularly simple: similarity corresponds to equality and transversality corresponds to $A\cap B =\{0\}$.  To exploit these facts, recall the concept of Cartan decomposition for a reductive Lie group: $G=KAK$, where $K$ is a maximal compact subgroup and $A$ is a Cartan subgroup.  Let $\pi:G\to \mathfrak{a}/W$ be the associated Cartan projection ($W$ is the Weyl group); it is continuous, surjective and proper.  Write $\mathfrak{a}(S)=\pi(S)W \subset \mathfrak{a}$ for any subset $S\subset G$.  The considerations above indicate that properness can be checked by examining $\mathfrak{a}(H)$ and $\mathfrak{a}(\Gamma)$.  Specifically, $H \transversal \Gamma$ (resp. $H \sim \Gamma$) in $G$ if and only if $\mathfrak{a}(H) \transversal \mathfrak{a}(\Gamma)$ (resp. $\mathfrak{a}(H) \sim \mathfrak{a}(\Gamma)$).  Another way to state the properness condition is that $\Gamma$ acts properly on $H\backslash G$ if and only if $\mathfrak{a}(\Gamma)$ goes away from $\mathfrak{a}(H)$ at infinity, i.e. for any compact subset $C$ of $\mathfrak{a}$, one has that $\mathfrak{a}(\Gamma) \cap (\mathfrak{a}(H)+C)$ is compact. The result is the following theorem, which strengthens Theorem \ref{cm} by providing a converse:

\begin{thm}[General Calabi--Markus phenomenon]
Let $H\backslash G$ be of reductive type.  Then the following are equivalent:
\begin{itemize}
	\item $\mathbb{R}$-rank$(H)$ = $\mathbb{R}$-rank$(G)$
	\item Only a finite $\Gamma$ acts properly discontinuously on $H \backslash G$.
\end{itemize}
\end{thm}
% OTHER PAPERS: wolf, kulkarni, kob89
 
In \cite{benoist}, Benoist develops a further use of the Cartan projection to put a restriction on any properly acting $\Gamma$.  Assume $H$ is stable under the Cartan involution defining $K$.  Define $B_W = \{aW \in \mathfrak{a}/W: aW = -aW\}$.  

%Let $\mathfrak{a}^+$ be the positive Weyl chamber in $\mathfrak{a}$ with respect to some choice of a positive system of (restricted) roots, and let $\mathfrak{a}^+(H)$ be the image of the Cartan projection of $H$ in this chamber.  Denote the \emph{opposition involution} on $\mathfrak{a}^+$ by sending $a$ to $\iota(a):=(\mbox{the unique point in }\mathfrak{a}^+\mbox{ conjugate under }W \mbox{ to } -a$).  Let $B^+:=\{a\in \mathfrak{a}^+ : \iota(a)=a\}.$  Benoist obtains the following theorem:

\begin{thm}[Benoist, \cite{benoist} \S7.5]\label{ben_thm}
Let $H\backslash G$ be of reductive type and let $H$ be stable under the Cartan involution defining $K$ as above.  Only a virtually abelian $\Gamma$ may act properly discontinuously on $H\backslash G$ if and only if $B_W \subset \mathfrak{a}(H)$.  In particular, if $H\backslash G$ is noncompact and $B_W \subset \mathfrak{a}(H)$, then $H\backslash G$ has no compact form.
\end{thm}

Table \ref{benoist_table} lists a few spaces to which this theorem applies.

\begin{table}[h]
\begin{center}
\begin{tabular}{|c|c|c|}
\hline
	$G$ & $H$ & conditions \\
	\hline
	$\SL{n}$ & $ (\SL{k}\times \SL{n-k})$ & $1\leq k\leq n$, and $k(n-k)$ even \\ 
	\hline
	$\SL{2n}$ & $\textrm{Sp}(n, \mathbb{R})$ & $n\geq 1$ \\
	\hline
	$\SL{2n}$ & $\textrm{SO}(n,n)$ & $n\geq 1$ \\ 
	\hline
	$\SL{2n+1}$  & $\textrm{SO}(n, n+1)$ & $n\geq 1$ \\
	\hline
	$\textrm{SO}(n+1, m)$ & $\textrm{SO}(n,m)$ & $n\geq m$, $n$ even\\
	\hline
	 &  & $G$ simple, complex;\\
	& &  $H_\mathbb{C}$ the fixed points of a \\ 
	$G_{\mathbb{C}}$& $H_{\mathbb{C}}$ & complex involution of $G_\mathbb{C}$; \\
	& &  except \\
	& & $(\textrm{SO}(4n-k, \mathbb{C})\times \textrm{SO}(k,\mathbb{C}))\backslash \textrm{SO}(4n, \mathbb{C})$ \\
	& &  for $n\geq 2$, and $k$ odd.  \\
	\hline 
\end{tabular}
	\caption{Some homogeneous spaces $H\backslash G$ without compact Clifford--Klein forms\label{benoist_table}}
\end{center}
\end{table}
Note that the first example in Table \ref{benoist_table}, with $k=1$, is $\SL{n-1}\backslash \SL{n}$ when $n$ is odd.\footnote{As reported in \cite{Witte_Iozzi}, the case $\SL 2\backslash \SL 3$, was also proven by Margulis, though never published.}  When $n$ is even, Benoist constructs non-virtually abelian groups which act properly discontinuously (though not cocompactly) on $\SL{n-1}\backslash \SL{n}$; specifically, he constructs nonabelian, free, Schottky groups.\footnote{Benoist's construction works for all reductive $H\backslash G$ admitting proper actions by non-virtually abelian $\Gamma$; he characterizes these in \cite{benoist}.}  The case for $n$ even is an interesting open question; one expects that no compact form of this space exists.

We will close this subsection with a mention of some of the work of Kassel on the problem.  She works on homogeneous spaces $H\backslash G$ with $\mathbb{R}$-$\textrm{rank}(H)=\mathbb{R}$-$\textrm{rank}(G)-1$, i.e. examples just outside the reach of the Calabi--Markus phenomenon.  In these cases (and under some natural assumptions on $G$ and $H$) she is able to describe the Cartan projection of any $\Gamma$ acting properly discontinuously on $H\backslash G$ (\cite{kassel_proper} Thm 1.1).  Two applications of the main result in \cite{kassel_proper} are a structure theorem for properly discontinuous $\Gamma$ for $diag(G')\backslash (G'\times G')$ (\cite{kassel_proper} Theorem 1.3; $diag$ refers throughout to the diagonal embedding), and a simplified proof of Benoist's nonexistence result for $\SL {n-1}\backslash \SL n$ for $n$ odd.  Kassel also extends this work to Lie groups over local fields.  We will return to some of Kassel's further work along this line in Section \ref{sec_salein_kassel}.

%
%%
%%%
%%%%
%%%%%%%%%%%%%%%%
\subsection{Results in small dimensions}

There are scattered results on the existence question when we restrict to small dimension.  I would like to mention here the case $\textrm{SL}_2(\mathbb{K})\backslash \textrm{SL}_3(\mathbb{K})$ where $\mathbb{K}$ is any of the real numbers, the complex numbers or the quaternions.  Benoist's work covers these cases; earlier, Kobayashi dealt with the latter two cases (see \cite{ko_disc}).  Oh and Witte-Morris generalized Benoist's work to homogeneous spaces of $\SL 3$ with $H$ nonreductive in \cite{Oh_Witte_compact}, proving that only the obvious examples have compact forms:

\begin{thm}[Oh--Witte-Morris, \cite{Oh_Witte_compact} Prop 1.10]
Let $H$ be a closed, connected subgroup of $\textrm{SL}_3(\mathbb{K})$.  If $H\backslash \emph{SL}_3(\mathbb{K})$ has a compact form, either $H$ or $H\backslash \emph{SL}_3(\mathbb{K})$ is compact.
\end{thm}

\noindent There is a nice treatment of this problem and a simplification of Oh and Witte-Morris's proof in Section 6 of \cite{Witte_Iozzi}.

%
%%
%%%
%%%%%%%%%%%%%%%
\subsection{A dimension criterion}

In the approaches to the problem just detailed, only the properness of the action of $\Gamma$ is considered.  In \cite{kobayashi-proper}, Kobayashi uses the ideas of Definition \ref{kob_defn} to examine conditions under which a compact form does exist.  He produces the following theorem.

\begin{thm}[\cite{kobayashi-proper}, Thm 4.7]\label{d_thm}
Let $G$ be a reductive linear group; $H$ and $L$ subgroups having finitely many components and stable under some Cartan involution $\theta$ of $G$.  If
\begin{itemize}
	\item $\mathfrak{a}(H)\cap \mathfrak{a}(L) = \{0\}$ and 
	\item $d(H)+d(L) = d(G),$
\end{itemize}
then both $H\backslash G$ and $G/L$ admit compact forms.
\end{thm}
The first condition of this theorem arises from the properness issues discussed above.  The second is a consideration for compactness.  Its necessity follows from the following Proposition.

\begin{prop}
Let $G$ be reductive and linear with connected closed subgroups $H\transversal L$.  Then $d(H)+d(L)\leq d(G)$ with equality if and only if $H\backslash G/L$ is compact.
\end{prop}

Table \ref{kob_table_1} lists some homogeneous spaces that \emph{do} admit compact forms under the criterion of Theorem \ref{d_thm} together with the relevant subgroup $L$.  We encounter here the compact forms that are evidence for Conjecture \ref{kob_conj}.  As we noted in Remark \ref{Oniscik_rmk}, up to switching $H$ and $L$, these are the only homogeneous spaces of reductive type with $G$ simple and linear admitting standard forms (up to conjugation of $H$ in $G$ and perhaps taking connected components).

\begin{table}
\begin{center}
\begin{tabular}{|c|c|c|}
\hline
	$G$ & $H$ & L \\
	\hline
	$\textrm{SU}(2, 2n)$ & $\textrm{Sp}(1,n)$ & $\textrm{U}(1,2n)$ \\ 
	\hline
	$\textrm{SO}(2,2n)$ & $\textrm{U}(1,n)$ & $\textrm{SO}(1,2n)$ \\
	\hline
	$\textrm{SO}(4,4n)$ & $\textrm{SO}(3,4n)$ & $\textrm{Sp}(1,n)$ \\ 
	\hline
	$\textrm{SO}(4,4)$  & $\textrm{SO}(4,1)\times \textrm{SO}(3)$ & $\textrm{Spin}(4,3)$ \\
	\hline
	$\textrm{SO}(4,3)$ & $\textrm{SO}(3,1)\times \textrm{SO}(2)$ & $\textrm{G}_{2(2)}$\\
	\hline
	 $\textrm{SO}(8,8)$  & $\textrm{SO}(7,8)$ & $\textrm{Spin}(1,8)$\\
	\hline 
	 $\textrm{SO}(8,\mathbb{C})$  & $\textrm{SO}(7,\mathbb{C})$ & $\textrm{Spin}(1,7)$\\
	\hline 
	 $\textrm{SO}(8,\mathbb{C})$  & $\textrm{SO}(7,1)$ & $\textrm{Spin}(7,\mathbb{C})$\\
	\hline 
	 $\textrm{SO}(8,\mathbb{C})$  & $\textrm{SO}(7,\mathbb{C})$ & $\textrm{Spin}(1,7)$\\
	\hline 
	$\textrm{SO}^*(8)$ & $\textrm{U}(3,1)$ & $\textrm{Spin}(1,6)$ \\
	\hline
	$\textrm{SO}^*(8)$ & $\textrm{SO}^*(6)\times \textrm{SO}^*(2)$ & $\textrm{Spin}(1,6)$ \\
	\hline
\end{tabular}
	\caption{Some homogeneous spaces $H\backslash G$ with compact Clifford--Klein forms\label{kob_table_1}}
\end{center}
\end{table}

The standard compact forms with $G=\textrm{SO}(2,2n)$ and $G=\textrm{SO}(4,4n)$ were first found by Kulkarni in \cite{kulkarni}; the rest of the examples in Table \ref{kob_table_1} are due to Kobayashi in \cite{kobayashi-proper}, \cite{ko_disc} and \cite{kobayashi_lecture}.

%
%%
%%%
%%%%%
%%%%%%%%%%%%%%%%%%%
\subsection{Virtual cohomological dimension}

The central thrust of the Calabi--Markus phenomenon and the related results above is that there must be some coherence among the sizes (measured by real rank or by dimension of the corresponding Riemannian symmetric space) of $G$, $H$, and $L$.  This idea was used further by Kulkarni in examining the spaces $\textrm{O}(p,q)\backslash \textrm{O}(p+1, q)$.  When $p$ or $q$ are 1, we recover the relativistic space forms Calabi and Markus studied in \cite{calabi-markus}.  Kulkarni proves a number of results in his article.  Toward our purposes here, he uses cohomological dimension of groups and a generalized Gauss-Bonnet formula to prove the following:

\begin{thm}[Kulkarni, \cite{kulkarni} Cor 2.10]
If $p$ and $q$ are odd, then there is no compact Clifford--Klein form of $\emph{O}(p,q)\backslash \emph{O}(p+1, q)$.
\end{thm}

This approach has found its most complete statement in the work of Kobayashi:

\begin{thm}[Kobayashi \cite{ko_necessary}, Thm 1.5]\label{sim_thm}
Let $H\backslash G$ be of reductive type.  If there exists a reductive $H' < G$ such that $H' \sim H$ and $d(H') > d(H)$, then $H\backslash G$ has no compact form.
\end{thm}

From the definition of $\sim$ (\ref{kob_defn}) it is easy to check that any discrete $\Gamma$ acting properly on $H\backslash G$ will also act properly on $H'\backslash G$.  Thus, we can construct a restriction on the sizes of $G$, $H$ and $\Gamma$ for proper actions in the vein of Theorem \ref{d_thm} but with the discrete group $\Gamma$ replacing $L$.  This is supplied by the following Lemma, in which virtual cohomological dimension ($vcd$) replaces the dimension of the corresponding symmetric space.  The proof of Theorem \ref{sim_thm} from this lemma is clear. Recall that the cohomological dimension of a group $\Gamma$ is the maximal $n$ for which the group cohomology $H^n(\Gamma,M)$, where $M$ is an arbitrary $\mathbb{R}\Gamma$-module, does not vanish.  The \emph{virtual} cohomological dimension is this value for any finite-index subgroup of $\Gamma$ (see \cite{brown_c_g}).

\begin{lemma}[Serre \cite{serre_cohomologie} when $H\backslash G$ is simply connected, Kobayashi \cite{ko_necessary} in general]
If discrete $\Gamma < G$ satisfies $\Gamma \transversal H$, then
\begin{itemize}
	\item $vcd(\Gamma)+d(H)\leq d(G)$ and
	\item $vcd(\Gamma) + d(H) = d(G)$ if and only if $\dq$ is compact.
\end{itemize}
\end{lemma}

This approach supplies numerous other examples of homogeneous spaces without compact forms.  The examples in Table \ref{kob_table_2} are taken from \cite{ko_necessary}, Table 4.4; see that table for a number of examples involving exceptional Lie groups.

\begin{table}[ht]
\begin{center}
\begin{tabular}{|c|c|c|}
\hline
	$G$ & $H$ & $\mathfrak{h}'$ \\
	\hline
	$\SL{2n}$ & $\textrm{SO}(n,n)$ & $\mathfrak{sp}(n, \mathbb{R})$ \\
	\hline
	$\textrm{SU}^*(2n)$ & $\textrm{SO}^*(2n)$ & $\mathfrak{sp}(\lfloor \frac{n}{2} \rfloor, n-\lfloor \frac{n}{2} \rfloor)$ \\
	\hline
	$\textrm{SU}(2n, 2n)$ & $\textrm{SO}^*(4n)$ & $\mathfrak{sp}(n,n)$ \\
	\hline
	$\textrm{Sp}(2n, \mathbb{R})$ & $\textrm{U}(n,n)$ & $\mathfrak{sp}(n, \mathbb{C})$ \\
	\hline
	$\textrm{SO}(2n, 2n)$ & $\textrm{SO}(2n, \mathbb{C})$ & $\mathfrak{u}(n,n)$ \\
	\hline
	$\textrm{SO}^*(2n)$ & $\textrm{SO}^*(2p) \times \textrm{SO}^*(2q)$ & $\mathfrak{so}(2) + \mathfrak{so}^*(2n-2)$ \\
	\hline
	$\textrm{SL}_n(\mathbb{C})$ & $\textrm{SO}(n, \mathbb{C})$ & $\mathfrak{u}(\lfloor \frac{n}{2} \rfloor, n-\lfloor \frac{n}{2} \rfloor)$ \\
	\hline
	$\textrm{SO}(n, \mathbb{C})$ & $\textrm{SO}(p, \mathbb{C})\times \textrm{SO}(q, \mathbb{C})$ & $\mathfrak{so}(n-1, \mathbb{C})$ \\
	\hline
	$\textrm{SO}^*(2n)$ & $\textrm{U}(p,q) $ & $\mathfrak{so}^*(2r)$ \\
	 & $n\geq 3, \quad 3p\leq 2n\leq 6p$ & $r=\min(n, 2p+1, 2q+1)$ \\
	\hline
	$\textrm{SL}_{2n}(\mathbb{C})$ & $\textrm{SU}(n,n)$ & $\mathfrak{sp}(n, \mathbb{C})$ \\
	\hline
	$\textrm{Sp}(n,\mathbb{R})$ & $\textrm{Sp}(p, \mathbb{R})\times \textrm{Sp}(q, \mathbb{R})$ & $\mathfrak{sp}(n,\mathbb{R})$\\
	\hline
\end{tabular}
	\caption{Some homogeneous spaces $H\backslash G$ without compact Clifford--Klein forms. Throughout, $n=p+q$ and $p,q>1$. \label{kob_table_2}}
\end{center}
\end{table}

%
%%
%%%
%%%%
%%%%%%%%%%%%%%%%%%%
\subsection{Clifford algebras}

As a final entry in this catalogue of `topological' approaches, let me record the following theorem of Kobayashi and Yoshino:

\begin{thm}[Kobayashi--Yoshino, \cite{ksurvey} Thm 4.2.1]
The following triples of Lie groups satisfy the conditions of Thm \ref{d_thm}, and hence $H\backslash G$ and $G/L$ have compact forms:

%\begin{table}[!h]
%\begin{center}
\begin{tabular}{|c|c|c|}
\hline
	$G$ & $H$ & $L$ \\
%	\hline
%	$GL(1, \mathbb{R})$ & $\{1\}$ & $Spin(1,1)$ \\
%	\hline
%	$Sp(1, \mathbb{R})$ & $\{1\}$ & $Spin(1,2)$ \\
%	\hline
%	$Sp(1, \mathbb{C})$ & $\{1\}$ & $Spin(1,3)$ \\
%	\hline
%	$Sp(1, 1)$ & $Sp(1)$ & $Spin(1,4)$ \\
	\hline
	$\textrm{GL}_2(\mathbb{H})$ & $\textrm{GL}_1(\mathbb{H})$ & $\textrm{Spin}(1,5)$ \\
	\hline
	$\textrm{O}^*(8)$ & $\textrm{O}^*(6)$ & $\textrm{Spin}(1,6)$ \\
	\hline
	$\textrm{O}(8, \mathbb{C})$ & $\textrm{O}(7, \mathbb{C})$ & $\textrm{Spin}(1,7)$ \\
	\hline
	$\textrm{O}(8,8)$ & $\textrm{O}(7,8)$ & $\textrm{Spin}(1,8)$ \\
	\hline
\end{tabular}
%\end{center}
%\end{table}

\end{thm}

Some of these have appeared already in Table \ref{kob_table_1}, and several can be obtained by other approaches.  What is particularly interesting in this theorem is the unified treatment of the $\textrm{Spin}(1,q)$ cases and the introduction of a new technique, namely, the use of Clifford algebras to calculate Cartan projections and verify the first condition of Theorem \ref{d_thm}.  I will not attempt to survey this approach here, but instead refer the reader to \cite{ksurvey}.

%
%%
%%%
%%%%
%%%%%
%%%%%%
%%%%%%%
%%%%%%%%%%%%%%%%%%%%%%%%%%%%%%%%%%%%%%%%%%%%%
%%GEOMETRY%%%%%%%%%%
\section{Geometry}\label{sec_geom}

%%%BENOIST-LABOURIE%%%%%%%
\subsection{A contribution from symplectic geometry}\label{ben_lab}

In \cite{bl}, Benoist and Labourie provide a useful restriction on $H$ via symplectic geometry.  The restriction is algebraic -- the center of $H$ cannot be `large' in that it cannot contain any hyperbolic elements.  

A simplified version of Benoist and Labourie's proof is available in the following special case.  Suppose $T$ is a one-parameter subgroup of semisimple $G$ generated by a hyperbolic element $t$ in $\mathfrak{g}$; let $Z_G(T)$ be the centralizer of $T$.  We will show that there is no compact manifold with a $(G, G/Z_G(T))$-structure.  first, note that $Z_G(T)$ is isomorphic to a direct product $M\times T$ for some subgroup $M$ of $G$.  Then one has a fiber bundle $G/M \to G/Z_G(T)$ with $\mathbb{R}$-fibers corresponding to the fibration by $T$-orbits.  The Killing form $\langle\phantom{r},\phantom{i} \rangle$ gives a 1-form on $G/M$, namely $Y\mapsto \langle t,Y\rangle$ for any $Y\in \mathfrak{g}$, which descends to $G/Z_G(T)$.  One can show that this 1-form is a connection whose curvature is a symplectic form, specifically $\omega_t(Y,Z)=\langle t, [Y,Z]\rangle$, for $Y, Z \in \mathfrak{g}$.  As $\omega_t$ is symplectic, $\omega_t^{\wedge d}$ is the volume form on $G/Z_G(T)$.  However, after perhaps passing to a finite cover, there is a non-vanishing section of the line bundle over $G/Z_G(T)$, and $\omega_t$ will be exact -- the differential of the connection 1-form associated to that section.  This in turn implies that the volume form obtained from $\omega_t$ is exact, which is impossible for a compact manifold and proves there is no compact form.  Note that this entire argument deals with local, $G$-invariant objects, so it can be transferred over to a manifold with a $(G, G/Z_G(T))$-structure.

%hence that the volume form arising from it is exact and hence trivial. This certainly cannot be true of a volume form, hence we arrive at a contradiction to the existence of a compact form of $G/H$ when $H$ is the full centralizer of a hyperbolic element. %The one can use $X$ to build a connection on $G/G_X$ such that the curvature tensor for this connection is a symplectic form.  The volume form supplied in the usual way by this symplectic form is trivial, proving that there is no compact for of $G/G_X$.  

Benoist and Labourie adapt this argument for the more general situation $H\subseteq Z_G(T)$, providing restrictions on the center of $H$.  Their argument involves using fibrations by weight spaces for elements in $H$ to build a connection with corresponding curvature tensor which is symplectic and cohomologically trivial.  They obtain the following results.  Note that as their approach deals only with local, $G$-invariant objects, Theorem \ref{blthm} and some of its corollaries apply to compact manifolds locally modeled on $H\backslash G$, not just to Clifford--Klein forms.

\begin{thm}[Benoist--Labourie, \cite{bl} Thm 1]\label{blthm}
Let $G$ be a connected, semisimple real Lie group, $H$ a connected unimodular subgroup.  If there exists a compact manifold with a $(G, H\backslash G)$-structure, then the Lie algebra of the center of $H$ does not contain any hyperbolic elements.
\end{thm}

\begin{cor}[Benoist--Labourie, \cite{bl} Cor 1]
Let $G$ be algebraic, semisimple and real. If $H$ is algebraic and reductive and there exists a compact manifold with a $(G, H\backslash G)$-structure, then the center of $H$ is compact.
\end{cor}

\begin{cor}[Benoist--Labourie, \cite{bl} Cor 2]
If $G$ is connected, semisimple, with finite center and $H$ is connected and reductive and $H\backslash G$ admits a compact Clifford--Klein form, then the center of $H$ is compact.
\end{cor}

%%%%REP THEORY%%%%%%%%%%%
\subsection{Two approaches via representation theory}\label{sec_rep_thry}

In \cite{margulis}, Margulis presents an approach to the existence question built on representation theory of the groups involved.  Specifically, he gives a criterion for nonexistence based on matrix coefficients for unitary representations of $H$ on certain $L^2$ functions on $H$.  

Let $G$ be unimodular, locally compact, $K$ a compact subgroup and $H$ a closed subgroup.  Matrix coefficients enter via his definition of $(G,K,H)$-tempered actions on some space $X$; for our purposes the following is the relevant application of this definition.

\begin{defn}[$(G,K)$-tempered subgroups, \cite{margulis} Defn 2]
We call $H$ a $(G,K)$-tempered subgroup of $G$ if there exists a function $q\in L^1(H)$ such that 
\[|\langle \psi(h)w_1, w_2 \rangle| \leq q(h) \Vert w_1 \Vert \Vert w_2 \Vert\]
for all $h\in H$, all $\psi(K)$-invariant functions $w_1, w_2 \in L^2(H)$ and  all unitary representations $\psi$ of $G$ on $L^2(H)$.\footnote{For $L^1(H)$ and $L^2(H)$ the implied measure is the Haar measure on $H$.}
\end{defn}

Margulis's result is the following:

\begin{thm}[Margulis, \cite{margulis} Thm 1]
Let $G, H$ and $K$ be as above, and let $F$ be any noncompact closed subgroup of $H$.  Suppose that $H$ is $(G, K)$-tempered.  Then there is no compact form of $F\backslash G$.  (In particular, there is no compact form of $H\backslash G$.)
\end{thm}

The main idea of Margulis's approach is the following. Suppose there is a compact form of $H\backslash G$.  Then one may take a compact set $M\subset G/\Gamma$ such that $HM = G/\Gamma$.  Margulis constructs a pair of $L^2$ functions, one supported around $M$ and one supported on some compact set far from $M$ in the $H$-direction.  Applying the $(G,K)$-tempered condition with these functions, he shows that the measure of $G/\Gamma-HM$ is positive; in particular that $HM$ cannot be all of $G/\Gamma$.

The $(G,K)$-tempered criterion is not always easy to check, but Margulis provides several nice examples.  Further examples and more detailed exposition can be found in \cite{oh_tempered}.  First, if $G$ is connected semisimple with Kazhdan's Property $(T)$ (see, e.g. \cite{zimmer-book} Chapt. 7), $K$ maximal compact and $H$ abelian and diagonalizable, standard arguments showing exponential decay of matrix coefficients can be applied to show $H$ is $(G,K)$-tempered.  Perhaps of more interest to us here, however, are two results involving the case $G=\SL n$, $K=\textrm{SO}(n)$.  By direct computations and some knowledge of representation theory, one can show that $H=\alpha_n(\SL 2)$ is $(G,K)$-tempered, where $\alpha_n$ is the irreducible $n$-dimensional representation of $\SL 2$ and $n\geq 4$.    Likewise, for $n\geq 3$, let $\phi:L \to \SL n$ be a representation of a connected simple Lie group, with irreducible components $\phi_i$.  If the sum of the highest weights for the nontrivial $\phi_i$ is larger than the sum of the positive roots for $L$, then $\phi(L)$ is $(G,K)$-tempered.  Thus, if there are many nontrivial $\phi_i$, or high-dimensional $\phi_i$ in comparison to the `size' of $L$ (as measured by the sum of its positive roots) then $\phi(L) \backslash \SL n$ has no compact form.

It is perhaps instructive to put these results alongside the results for the examples $\SL {n-k}\backslash \SL n$ discussed in connection with Kobayashi and Benoist's results (in particular Thm \ref{ben_thm}).  For Margulis's approach, we need $\SL{n-k}$ to be quite small, but to be embedded in $\SL n$ in a `large' way -- irreducibly, perhaps.  Benoist's result holds when $\SL{n-k}$ is large ($k=1$) but is not embedded irreducibly in $\SL n$.

To close this section on `geometric' approaches to the problem, let us note a second approach that uses representation theory.  This approach is due to Shalom \cite{shalom_rigidity}, and again relies on decay of matrix coefficients, although in a different way than Margulis.  Shalom studies unitary representations of $\textrm{SO}(1,n)$ and $\textrm{SU}(1,n)$ and their cohomology.  He is able to reproduce some of the rigidity results associated to groups of higher rank in these cases.  A full description of his results is far beyond our present scope,  but presenting their application to the compact forms question will serve as a bridge between the representation-theoretic approach and the dynamical approach to be described in the next section.

Shalom proves the following theorem.  

\begin{thm}[Shalom, \cite{shalom_rigidity} Thm 1.7]
Let $G$ be a simple Lie group with finite center, $\Lambda$ a discrete, infinite subgroup of $G$ admitting a discrete embedding into $\emph{SO}(1,n)$ or $\emph{SU}(1,n)$.  Assume there exists a nonamenable, closed subgroup $L$ in $G$ with noncompact center which commutes with $\Lambda$.  Then $\Lambda \backslash G$ admits no compact quotient.
\end{thm}

Taking $H$ a closed Lie subgroup of $G$ containing $\Lambda$ yields results on compact forms of homogeneous spaces.  One particular example is the following:

\begin{cor}\label{shalom_cor}
Let $G=\SL{n}$, $n\geq 4$, $H=\SL 2$ embedded in the upper left-hand corner of $G$.  Then $H\backslash G$ has no compact quotient.  The same holds with $\mathbb{C}$ replacing $\mathbb{R}$.
\end{cor}

By way of explaining how this theorem serves as a bridge between our exposition of the approach via representation theory and the approach via dynamics, let us give the briefest of overviews of the proof.  Following an idea developed by Zimmer, and which will be used repeatedly in the next section, Shalom studies the $L$-action by left-multiplication on the $\Gamma$ bundle $G/\Gamma \to \Lambda\backslash G/\Gamma$ over a compact form.  Shalom then uses the representation theory he has developed to study cocycles associated to this action to complete the proof.  A similar idea will appear again in the next section, where dynamics will fully take over with the appearance of ergodic theory.

%
%%
%%%
%%%%
%%%%%
%DYNAMICS%%%%%%%%%%
\section{Dynamics}\label{sec_dynamics}

%
%%
%%%%%%%
\subsection{The work of Zimmer, Labourie and Mozes}\label{sec_lmz}

The approach to the compact forms problem via dynamics was pioneered by Zimmer in \cite{zimmer1}.  In this paper he deals with cases where $Z_G(H)$ is a large group; in particular, he works under the following assumption:
\begin{itemize}
	\item There is a semisimple group of higher rank $J<Z_G(H)$.
\end{itemize}
An illustrative example of this case is given by taking $G=\SL{n}$ and $H=\SL{n-k}$ with $k\geq 3$, embedded in the upper left-hand corner of $G$.  Here $J=\SL{k}$ and the assumption above requires only that $k\geq 3$.

The key element to Zimmer's approach is the application of his cocycle superrigidity theorem.  Specifically, $J$ acts by left-multiplication on the $H$-bundle $G/\Gamma \to \dq$.  If $\sigma: \dq \to G/\Gamma$ is a measurable section of this bundle, $j\in J$ and $x\in \dq$, one calculates that

\begin{equation}\label{lifteqn}
	\sigma(j\cdot x) = \alpha(j, x) j \cdot \sigma(x)
\end{equation}
where $\alpha: J \times \dq \to H$ is a measurable map.  It is easy to verify that $\alpha$ satisfies the (dynamical) cocycle equation $\alpha(j_1j_2, x) = \alpha(j_1, j_2\cdot x)\alpha(j_2,x)$.  Zimmer's cocycle rigidity theorem (see \cite{zimmer-book}) states that if $J$ is semisimple of higher rank, $H$ is an algebraic group, the algebraic hull of $\alpha$ in $G$ is connected semisimple, and the $J$-action on $\dq$ is irreducible, then the cocycle $\alpha$ is equivalent to a trivial cocycle.  That is, up to a new choice of the measurable section, $\alpha$ is independent of the space variable $x$, in which case the cocycle equation reduces to the assertion that $\alpha$ is a homomorphism from $J$ to $H$.  

Two remarks are in order here.  First, the \emph{algebraic hull} is the unique (up to conjugacy) minimal algebraic group in which a cocycle equivalent to $\alpha$ takes values.  Questions about the algebraic hull of $\alpha$ are dealt with in the final section of \cite{lmz}.  The assumptions on the hull detailed above are ensured, perhaps after moving to a finite ergodic cover of $\dq$, provided that the algebraic hull is not compact.  In the compact case, simpler arguments are deployed in \cite{lmz}.   Second, the action of $J$ on $\dq$ is \emph{irreducible} for a $J$-ergodic probability measure $\mu$ if the action by any noncentral normal subgroup of $J$ is also ergodic.  This is a significant issue for the compact forms question; it has thus far been addressed by assuming that all simple factors of $J$ are higher rank on their own so that superrigidity applies to each individually.  See Question \ref{soquestion} below for a comment on a situation in which the irreducibility issue obstructs the solution of a natural compact forms question by these methods.

Zimmer first approaches the compact forms question in \cite{zimmer1} and continues with Labourie and Mozes in \cite{lmz}.  Their most general result is the following:

\begin{thm}[Labourie--Mozes--Zimmer, \cite{lmz} Thm 1.1]\label{lmzthm}
Let $H$ and $G$ be real algebraic and unimodular, and suppose that there is a semisimple $J<Z_G(H)$ all of whose simple factors are higher rank.  Suppose that $J$ is not contained in a proper, normal subgroup of $G$ and that
\begin{itemize}
	\item[(i)] the image of every nontrivial homomorphism $\tilde J \to H$ has compact centralizer in $H$, where $\tilde J$ is the universal cover of $J$;
	\item[(ii)] there is a nontrivial, $\mathbb{R}$-split, 1-parameter subgroup $B$ in $Z_G(HJ)$ that is not contained in a normal subgroup of $G$.
\end{itemize}
If there is a compact form $\dq$, then $H$ is compact.
\end{thm}

In \cite{lmz}, Labourie, Mozes and Zimmer actually address a more general situation than Clifford--Klein forms.  Theorem \ref{lmzthm} applies to any compact manifold locally modeled on $H\backslash G$, the variant of Question \ref{exist_quest} mentioned in the introduction.

The proof of Theorem \ref{lmzthm} falls into two cases, as dictated by condition \emph{(i)}: either $\rho$ is trivial (i.e. the cocycle for the $J$-action takes values in a compact subgroup of $H$), or its image has compact centralizer.  In the second case a short argument allows the authors to consider the cocycle for the $B$-action (utilizing \emph{(ii)}) and conclude that \emph{it} takes values in a compact group (now $Z_H(\rho(J)$).  In either case one lifts the volume measure $m$ from $\dq$ to $G/\Gamma$ using the section $\sigma$ and averages it over a compact set in the fiber direction to obtain a finite, $J$-invariant measure which must by construction be the restriction of Haar measure on $G/\Gamma$ to a positive measure subset.  An application of Moore's ergodicity theorem (see, e.g. \cite{zimmer-book}) shows that this can only occur if $H$ is compact.

The conditions \emph{(i)} and \emph{(ii)} of Theorem \ref{lmzthm} are specific to details of the proof but not to the overall philosophy of applying superrigidity to the compact forms question.  Note that in the test case of $\SL{n-k}\backslash \SL{n}$, condition \emph{(i)} restricts application of the theorem to $k\geq n/2$.  In \cite{lz}, Labourie and Zimmer provide a separate argument to prove nonexistence for the case $k\geq 3$:

\begin{thm}[Labourie--Zimmer, \cite{lz}]\label{lzthm}
There is no compact form of $\emph{SL}_{n-k}(\mathbb{R})\backslash \emph{SL}_{n}(\mathbb{R})$ for $k\geq 3$.
\end{thm}

\begin{proof}[Sketch of proof]
We prove by contradiction; $G=\SL n$ and $H = \SL {n-k}$. Cocycle superrigidity provides a trivial cocycle for the $J=\SL{3}$-action -- i.e. $\alpha=\rho$, a homomorphism from $J$ to $G$, up to an error in a compact subgroup $K$ of $H$.  Lift the ergodic measure $\mu$ to a finite measure $\sigma_*(\mu)$ on $G/\Gamma$ and average this over $K$ to provide a finite measure $\hat\mu$ which covers $\mu$.  Examining equation \eqref{lifteqn} one sees that $\hat\mu$ is invariant under $gr(\rho)(J)$, the graph of $\rho$ applied to $J$.  Likewise, $h_*\hat\mu$ is invariant under $h\phantom{.}gr(\rho)(J)h^{-1}$, for $h\in H$.  If $\rho$ is irreducible as a representation of $\SL{3}$ into $\SL{n-k}$, Labourie and Zimmer fall back on the arguments of \cite{lmz}; if it is reducible they show that a properly chosen $H$-conjugate of $gr(\rho)$ intersects $H$ in a noncompact subgroup.  This noncompact group acting in the fiber direction of the bundle $G/\Gamma \to \dq$ cannot preserve the finite measure $h_*\hat\mu$, providing a contradiction.
\end{proof}

The final step of this argument involves algebraic structure that is fairly specific to the case of $\textrm{SL}$'s -- namely that there is an element of the Weyl group for $\SL{n}$ that exchanges any two diagonal entries while leaving the rest fixed.  To remove the algebraic conditions of Theorem \ref{lmzthm} it is necessary to adapt a more general approach as below.  It also proceeds by (eventually) reducing the question to a question about subgroups of $G$.

%
%%
%%%%%%%
\subsection{A recent improvement}\label{sec_me}

The following theorem presents an improvement to Theorem \ref{lmzthm} in that it removes the algebraic conditions $(i)$ and $(ii)$ and the requirement that $J$ is not contained in a proper, normal subgroup of $G$.  To achieve this, we must sacrifice slightly by requiring $G$ simple, but this is not a terrible loss.  Indeed, all examples Labourie, Mozes and Zimmer provide are for $G$ simple.

\begin{thm}[Constantine, \cite{compact_forms} Main Theorem]\label{mythm}
Let $G$ be a connected, simple Lie group with finite center and $H$ a connected, noncompact, reductive Lie subgroup. Suppose that $J<Z_G(H)$ is a simple Lie group with real rank at least two. Then there is no compact form of $H\backslash G$.
\end{thm}

\begin{proof}[Sketch of proof]
Proceed as above, applying cocycle superrigidity and producing the $gr(\rho)$-invariant measure $\hat\mu$ on $G/\Gamma$.  An additional result of cocycle superrigidity is that $\rho$ is a rational map, and hence it takes unipotent subgroups to unipotent subgroups.  We may assume that $J$ is generated by unipotents, then take an ergodic component of $\hat\mu$ for the $gr(\rho)$-action and apply Ratner's measure classification for unipotent flows to this measure (\cite{ratner}).  Ratner provides that this measure is the image of the Haar measure for some subgroup $L\supset gr(\rho)$ of $G$ along an $L$-orbit in $G/\Gamma$.

It is clear from its construction that the measure described by the $L$-orbit covers $\mu$ and its support extends only compactly in the $H$-fiber direction.  The latter fact implies that $L\cap H$ is compact.  To utilize the former, one studies the dynamics of the $J$-action on $\dq$ and the measure $\mu$.  The application of superrigidity allows one to calculate Lyapunov exponents for this action, recording exponential expansion and contraction of orbits under various flows in the $J$-action.  Standard tools from the theory of hyperbolic dynamics imply that (roughly speaking) the support of $\mu$ extends in directions with nonzero exponents.  Separate arguments using dynamics of unipotent elements in $J$ and the pseudo-Riemannian structure of $\dq$ show the support also extends in directions with zero exponents.  The result of all this is that since the $L$-orbit covers a measure whose support extends in all directions in $\dq$, one also has that $HL=G$.  The proof is completed by showing that under the conditions of the theorem, no subgroups $L$ satisfying $L\cap H$ compact and $HL=G$ exist (see Remark \ref{Oniscik_rmk}).
\end{proof}

One interesting feature of this method of proof is that it proceeds through verifying Conjecture \ref{kob_conj} for the homogeneous spaces under consideration; this holds with slightly loosened restrictions on $G$:

\begin{thm}[Constantine, \cite{compact_forms} Characterization Theorem]\label{class_thm}
Let $G$ and $H$ be as above, but with $G$ allowed to be semisimple rather than simple. Assume that there is a semisimple Lie group $J < Z_G(H)$ such that:\begin{itemize}
	\item[(1)] All simple factors of $J$ have real-rank at least two
	\item[(2)] The vector space sum of $\mathfrak{h}$ and the Lie algebra generated by all nonzero weight spaces for a Cartan subgroup $A<J$ is $\mathfrak{g}$.
\end{itemize}
Then any compact form of $H\backslash G$ is standard. 
\end{thm}

Table \ref{dynamics_table} lists a few examples of homogeneous spaces which these theorems imply do not have compact forms.  In their full generality, most arise from Theorem \ref{mythm}, with the exception of $\SL{n-k}\backslash \SL{n}$, which is due to Labourie and Zimmer as noted above.  Many are also proven by Kobayashi as noted in \cite{kobayashi_lecture}, but with stronger restrictions on $k$ and $l$.

\begin{table}
\begin{center}
\begin{tabular}{|c|c|c|}
\hline
	$G$ & $H$ & conditions \\
	\hline
	$\SL{n}$ & $\SL{n-k}$ & $k\geq 3$ \\ 
	\hline
	$\textrm{SL}_{n}(\mathbb{C})$ & $\textrm{SL}_{n-k}(\mathbb{C})$ & $k\geq 3$ \\
	\hline
	$\textrm{SO}(n, m)$ & $\textrm{SO}(n-k,m-l)$ & $k\geq 2, l\geq 3$ \\ 
	\hline
	$\textrm{PSO}(2n,\mathbb{C})$  & $ \textrm{PSO}(2(n-k),\mathbb{C})$ & $k\geq 2$ \\
	\hline
	$\textrm{SO}(2n+1, \mathbb{C})$ & $\textrm{PSO}(2n-k), \mathbb{C})$ & $k\geq 2$\\
	\hline
	$\textrm{SO}(2n+1, \mathbb{C})$ & $\textrm{SO}(2(n-k)+1, \mathbb{C})$ & $k\geq 2$ \\
	\hline
	$\textrm{SU}(p,q)$ & $\textrm{SU}(p-k, q-l)$ & $k,l\geq 2$ \\
	\hline
	$\textrm{Sp}(2m,\mathbb{R})$ & $\textrm{Sp}(2(m-k), \mathbb{R})$ & $k\geq 2$ \\
	\hline
	 &  & $H'$ a noncompact \\
	$G$ listed above & $H'$ & reductive subgroup of the \\
	& &  corresponding $H$ listed above \\ 
	\hline 
\end{tabular}
	\caption{Some homogeneous spaces $H\backslash G$ without compact Clifford--Klein forms\label{dynamics_table}}
\end{center}
\end{table}

To close this section, the author would like to pose the following question:

\begin{ques}\label{soquestion}
Does $\textrm{SO}(n-2, m-2)\backslash \textrm{SO}(n,m)$ have a compact Clifford--Klein form?
\end{ques}
The reader will notice that this case is excluded from Table \ref{dynamics_table}.  This is because the semisimple part of $Z_G(H)$ is $\textrm{SO}(2,2)$ which is semisimple and of higher rank, but not simple ($\textrm{SO}(2,2)_o \cong (\SL 2 \times \SL 2)/\{\pm (1,1)\})$.  To prove nonexistence of a compact form, one needs to address the irreducibility condition for superrigidity -- if this is accomplished, the argument which proves Theorem \ref{mythm} will apply.

%
%%
%%%
%%%%
%%%%%
%%%%%%
%%%%%%%
%%%%%%%%
%%%%%%%%%
%%%%%%%%%%
%%%%%%%%%%%
%%%%%%%%%%%%%%%%%%%%%%%%%%%%%%%%%%%%%%%%%%%%%%

\section{Deformations and moduli spaces of compact forms}\label{sec_deformations}

We close this survey by taking up some results, many of them quite recent, on the second question of the introduction.  Namely, when a homogeneous space has a compact form, what can we say about the space of all possible compact forms?  We will begin by collecting the evidence for Kobayashi's Conjecture \ref{kob_conj}.

%%%%%
\subsection{Evidence for the `standard forms' conjecture}

The reader will note that the only positive results on the existence question which we have seen so far are those provided by Table \ref{kob_table_1} and due to the algebraic construction of Theorem \ref{d_thm}.  That is, they are all standard forms.  This fact is the main empirical evidence for Conjecture \ref{kob_conj}.

We can note a small amount of further evidence.  It is (trivially) true in the Riemannian case. It is true for homogeneous spaces of $\textrm{SO}(2, n)$ by work of Oh, Witte-Morris and Iozzi which will be reported on below. Theorem \ref{class_thm} states that the conjecture is true in the stronger sense that all compact forms are standard when there is a higher-rank action present (and the rest of the requirements of that theorem are fulfilled).  Note, however, that the (now purely algebraic question) of whether any such forms exist is still open\footnote{Remark \ref{Oniscik_rmk} does not apply when $G$ is only semisimple.}, and their existence may be unlikely, given the situation when $G$ is simple.

This evidence is slight, of course; the main argument for the conjecture is the empirical one.  It seems to the author that this will be an extremely difficult conjecture to prove in general, in large part because, as we have seen, there is no over-arching approach to the problem.  

%
%%
%%%
%%%%
%%%%%
%%%%%%
%%%%%%%
%%%%%%%%
%%%%%%%%%
%%%%%%%%%%
%%%%%%%%%%%%%%%%%%%%%%%%%%%%%%%%%%%%%%%%%
\subsection{Moduli spaces of compact forms}

The reader will have noticed that Kobayashi's conjecture is not that every compact form is standard.  This simpler situation is ruled out by the following result, arrived at by Goldman, Ghys and Kobayashi.

\begin{thm}[\cite{goldman_nonstd}, \cite{ghys2}, \cite{kobayashi_def}]
There are nonstandard compact forms.
\end{thm}

There are a number of proofs to this theorem stated as such.  Let us begin with Goldman's and Kobayashi's original approaches and then proceed to look at some very recent results in this direction.  Ghys's work deals with the cases of $G= \SL 2$ and $\textrm{SL}_2(\mathbb{C})$.  He constructs some interesting examples.

%
%%
%%%
%%%%%%%%%%%%%%%
\subsubsection{The work of Kobayashi and Goldman}\label{sec_goldman}

Recall that we have previously defined $\textrm{AdS}^3$ as $\mathbb{H}^{2,1}=\textrm{SO}(2,2)/\textrm{SO}(1,2)$.  It can also be defined as $(\textrm{P}\SL 2 \times \textrm{P}\SL 2)/diag(\textrm{P}\SL 2)$ up to a covering of order 2, or as $\textrm{P}\SL 2$ endowed with the metric given by its Killing form, at least up to a finite cover.  The second description is more common among the authors whose work is discussed in this subsection and fits better with their generalizations of work on $\textrm{AdS}^3$, so I will adopt it now.  The description as $\textrm{P}\SL 2$ is favored by Salein, Kassel and Gu\'eritaud below.

In \cite{goldman_nonstd}, Goldman shows that not all compact forms of $\textrm{AdS}^3$ are standard, an issue raised by the work of  Kulkarni and Raymond in \cite{kul_ray}.  In this case, the question of whether all forms are standard becomes whether all forms have $\Gamma$ conjugate into $\SL 2 \times \textrm{SO}(2)$ or $\textrm{SO}(2) \times \SL 2$.  He proves

\begin{thm}[Goldman, \cite{goldman_nonstd} Thm 1]
Let $M$ be a standard compact form of $\emph{AdS}^3$ with $H^1(M; \mathbb{R})\neq 0$.  Then there is a nontrivial deformation space of nonstandard compact form structures on $M$.
\end{thm}

Goldman recasts compact forms of $\SL 2$ in the language of geometric structures.  Let $G$ be a Lie group and $X$ a homogeneous space of $G$.  As noted earlier, a $(G,X)$-structure on a manifold $M$ consists of a holonomy homomorphism $hol:\pi_1(M)\to G$ and a developing map $dev:\tilde M \to X$ which is a local diffeomorphism and which is equivariant with respect to $hol$.  The pair $(hol, dev)$ is well-defined up to the natural $G$-action: namely, $g\cdot(hol, dev) = (ghol(\cdot)g^{-1}, g\cdot dev)$.  The developing map provides a well-defined $X$-coordinate patch structure on $M$.  If $G$ preserves a pseudo-Riemannian metric on $X$ and we further require that the developing map be a local isomorphism; this enforces a unique pseudo-Riemannian metric on $M$.  The $(G,X)$-structure is \emph{complete} if the developing map is a covering map, in which case the $(G,X)$-structure gives $M$ the structure $\tilde X/\Gamma$ where $\Gamma = \pi_1(M)$ and $\tilde X$ is the universal cover of $X$.\footnote{This completeness is equivalent to a notion of geodesic completeness; see \cite{goldman_nonstd}.}

In this setting, let us take $X=\SL 2$ and $G = \SL 2 \times \SL 2$, acting by left- and right-multiplication, which preserves the metric given by the Killing form.  To find nonstandard forms, we search for holonomy representations that take unbounded image in both factors of $G$.  It is a general fact about $(G,X)$-structures -- the Ehresmann--Thurston principle, see \cite{thurston_notes} and \cite{ehresmann} -- that given any holonomy representation $h_0 \in \textrm{Hom}(\Gamma, G)$ for $\Gamma$ the fundamental group of a fixed compact manifold, there is an open set $U$ containing $h_0$ in the variety  $\textrm{Hom}(\Gamma, G)$ consisting of holonomy representations.  Let us take $h_0$ a holonomy representation of the form $(1, \pi)$ for some $\pi:\Gamma \to G$.  The $(G,X)$-structure defined by $h_0$ is complete and identifies $M$ with a quotient of $X$.  What remains for us to show is that we can find a nearby representation $h$ which still provides a \emph{free and properly discontinuous} action of $\Gamma$ on $X$.  

Goldman proves that $h$ of the following form works.  Let $B$ be any hyperbolic or parabolic one-parameter subgroup of $\SL 2$ and let $v$ be a nonconstant representation of $\Gamma$ into $B$ which is sufficiently close to the constant representation (which exists because of the assumption $H^1(M;\mathbb{R})\neq 0$). Then $h(\gamma)=(v(\gamma), h_0(\gamma))$ gives a free and properly discontinuous action and hence a $(G,X)$-structure on a compact manifold $M$ which is nonstandard.

In \cite{kobayashi_def}, Kobayashi extends Goldman's work significantly.  We record a definition first:

\begin{defn}
A deformation $\phi_t(\Gamma)$ of $\Gamma < G$ is called trivial if each $\phi_t(\Gamma)$ is conjugate to $\Gamma$ in $G$.  If all sufficiently small deformations of $\Gamma$ are trivial, $\Gamma$ is called locally rigid in $G$.
\end{defn}

\begin{thm}[Kobayashi, \cite{kobayashi_def} Thm A]\label{kob_def_thm_1}
Let $H \backslash G = diag(G')\backslash (G'\times G')$ where $G'$ is a simple linear Lie group.  Then
\begin{itemize}
	\item[(1)] For any uniform lattice $\Gamma' < G'$, the quotient $H\backslash G/\Gamma$ remains a compact form for $\Gamma$ any sufficiently small deformation of $\Gamma'\times \{1\}$ in $G$.  That is to say, any sufficiently small deformation of $\Gamma' \times \{1\}$ still acts properly discontinuously and cocompactly on $H\backslash G$.
	\item[(2)] It is possible to find uniform lattices with nontrivial deformations of this type in $G$ if and only if $G'$ is locally isomorphic to $\emph{SO}(1,n)$ or $\emph{SU}(1,n)$.
\end{itemize}	
\end{thm}

\noindent This generalizes Goldman's work by taking $G'=\textrm{PSL}_2(\mathbb{R}) \cong \textrm{SO}(1,2)_o$ and Ghys's by taking $G' = \textrm{PSL}_2(\mathbb{C}) \cong \textrm{SO}(1,3)_o$.  To prove $(2) \implies (1)$, Kobayashi proves that for $G'$ locally isomorphic to $\textrm{SO}(1,n)$ or $\textrm{SU}(1,n)$, there are uniform lattices in $G'\times 1$ with nontrivial deformations, and any sufficiently small deformation preserves the properly discontinuous character of the action on $diag(G')\backslash (G'\times G')$.  That $(1) \implies (2)$ follows from Weil's local rigidity theorem \cite{weil} and some vanishing theorems for Betti numbers (see citations in \cite{kobayashi_def}, p. 406).

Kobayashi also observes the following:

\begin{prop}[Kobayashi, \cite{kobayashi_def} Thm B and \S 1.8]\label{kob_def_result}
The following homogeneous spaces admit compact forms that have nontrivial deformations $(n\geq 1)$:
\[\emph{SO}(1,2n)\backslash \emph{SO}(2,2n), \emph{Sp}(1,n)\backslash \emph{SU}(2,2n),\] 
\[\emph{G}_2(\mathbb{R})\backslash \emph{SO}(3,4), \emph{Spin}(3,4)\backslash \emph{SO}(4,4).\]

There are locally rigid standard forms for the following homogeneous spaces $(n\geq 1, m\geq 2)$:
\[\emph{U}(1,2n)\backslash \emph{SU}(2,2n), \emph{SO}(3,4m)\backslash \emph{SO}(4,4m), \emph{Sp}(1,n)\backslash \emph{SO}(4, 4n).\]
\end{prop}

\noindent Note that all these examples are taken from Table \ref{kob_table_1}, which lists homogeneous  spaces with compact forms constructed from a triple of Lie groups $G,H,L$ ($H$ and $L$ play symmetric roles in this construction).  The key to Theorem \ref{kob_def_thm_1} and to the first set of examples in Proposition \ref{kob_def_result} is the existence of a nontrivial centralizer of $L$ in $G$.  When $Z_G(L)$ is nontrivial, an embedding of some lattice $\Gamma$ in $L$ into $G$ can be deformed by finding a nonconstant homomorphism $\rho:\Gamma \to Z_G(L)$ and mapping $\gamma \mapsto \gamma \rho(\gamma)$.  Denote these deformed embeddings of $\Gamma$ by $\Gamma_\rho$; Kobayashi shows that for $\rho$ in a sufficiently small neighborhood of the constant homomorphism, $\Gamma_\rho$ still acts properly discontinously, thus yielding a compact form.  In Proposition \ref{kob_def_result} these centralizers are compact and hence the deformations do not affect proper discontinuity.  In Theorem \ref{kob_def_thm_1} on the other hand, the centralizer is noncompact, making the proof that proper discontinuity of the action survives much more difficult.  For the second set of examples in Prop. \ref{kob_def_result} Kobayashi observes that any uniform lattice in the $L$ corresponding to the given pair is locally rigid in $G$.  One feature of this circle of results is that the deformation of $\Gamma$ does not need to be by the specific one-parameter subgroups Goldman uses.

In addition to the papers of Goldman and Kobayashi, we note that similar deformations have also been given by Ghys \cite{ghys1, ghys2} and Salein \cite{salein_killing}.  We will have more to say about Salein's extension of this work below.

%
%%
%%%
%%%%
%%%%%
%%%%%%
%%%%%%%%%%%%%%%%%%%%%

\subsubsection{The work of Salein, Kassel and Gu\'eritaud}\label{sec_salein_kassel}

Recently Salein, Kassel and Gu\'eritaud have continued the work done on moduli spaces of compact forms by providing new sufficient conditions for understanding when deformed embeddings of $\Gamma$ in $G$ still give rise to compact forms.  In this work, the new compact forms presented are no longer small, continuous deformations of standard forms.  Rather, they are far from the standard examples -- even topologically different.  Salein's demonstration of this fact was surprising and the work of these authors has greatly increased our understanding of moduli spaces of compact forms.

\vspace{.5cm}

Salein's work on the problem can be found in \cite{salein_exotic}.  He again studies $\textrm{AdS}^3$ which is a model space for all Lorentzian 3-manifolds of constant curvature -1 and is identified with $\textrm{P}\SL 2$ with the Lorentzian metric given by its Killing form.  Recall that $G:=\textrm{P}\SL 2 \times \textrm{P}\SL2$ acts on this space isometrically.  We will write $(\rho, \rho_0)(\Gamma)$ for the embedding of a discrete group in $G$, and view a compact form of this homogeneous  space as $\textrm{P}\SL 2/(\rho, \rho_0)\Gamma$ where $\Gamma$ acts on $\textrm{P}\SL 2$ by $(\rho(\gamma), \rho_0(\gamma))x = \rho(\gamma)x\rho_0(\gamma)^{-1}.$

Now let $\rho_0(\Gamma)$ be a Fuchsian group in $\textrm{P}\SL 2$; let $g$ be the genus of the usual quotient $\textrm{P}\SL 2/\rho_0(\Gamma)$.  We will call a representation $\rho$ from $\textrm{Hom}(\Gamma, \textrm{P}\SL 2)$ $\rho_0$-\emph{admissible} if $(\rho, \rho_0)(\Gamma)$ acts properly discontinuously on $\textrm{P}\SL 2$.\footnote{Kulkarni and Raymond prove for any pair of representations that if $(\rho, \rho_0)(\Gamma)$ acts properly on $\SL 2$ then either $\rho$ or $\rho_0$ is injective with discrete and cocompact image \cite{kul_ray}.}  Salein proves the following:

\begin{thm}[Salein, \cite{salein_exotic} Thm 2.1.1]
The set $Adm(\rho_0)$ of $\rho_0$-admissible homomorphisms is an open subset of $\emph{Hom}(\Gamma, \emph{PSL}_2(\mathbb{R}))$.  For certain choices of $\rho_0$, $Adm(\rho_0)$ is disconnected, and has components in every connected component of $\emph{Hom}(\Gamma, \emph{PSL}_2(\mathbb{R}))$, except the two extremal ones (which are copies of Teichm\"uller space).
\end{thm}

\noindent The openness part of this theorem is a consequence of the completeness of constant curvature compact Lorentz manifolds, which is due to Klingler \cite{klingler}, and of the Ehresmann--Thurston principle.  The second statement is Salein's contribution, and was quite surprising at the time in that it shows that the moduli space of compact forms is not connected.  In particular, there are nonstandard forms which are not deformations of standard ones.

The key step in Salein's approach is the following Lemma, which provides a criterion for admissibility of the pair $(\rho_0, \rho)$:

\begin{lemma}[Salein, \cite{salein_exotic}, Lemma 2.1.3]\label{salein_criterion}
If there exists a function $f:\mathbb{H}^2 \to \mathbb{H}^2$ which is uniformly strictly contracting and $(\rho_0, \rho)(\Gamma)$ equivariant in that
\[f(\rho_0(\gamma)z) = \rho(\gamma)f(z) \mbox{ for all } \gamma\in \Gamma, z\in \mathbb{H}^2,\]
then $\rho$ is $\rho_0$-admissible.
\end{lemma}

\noindent Techniques in Fuchsian groups and hyperbolic geometry, such as studying isometries via translation length, underpin Salein's approach to the problem.

Salein's criterion allows him to construct his examples in the following simple way.  Position a regular $4g$-gon which is a fundamental domain for $\rho_0(\Gamma)$ about the center of the Poincar\'e disk model for $\mathbb{H}^2$.  Let $f$ be a contraction of the fundamental domain towards this center, chosen so that the $4g$-gon maps to a $4g$-gon with geodesic sides, but now with angle sum $2\pi m$ with $1<m<2g$.  One can extend $f$ to be $\Gamma$-equivariant so that it satisfies Lemma \ref{salein_criterion}; the new angle sum implies the Euler number of an associated surface, which specifies which component of $\textrm{Hom}(\Gamma, \textrm{P}\SL 2)$ the representation belongs to.

Gu\'eritaud and Kassel have also studied criteria for admissible pairs in \cite{g-k1}, building on Kassel's earlier work in \cite{fanny_thesis}.  They study $(\rho_0, \rho)$-equivariant maps $f$ such as those Salein's criterion calls for, but in much greater generality.  They generalize to the isometry group of $n$-dimensional hyperbolic space and define for each pair $(\rho_0, \rho)$ of representations into $\textrm{Isom}(\mathbb{H}^n)=\textrm{PSO}(1,n)$ the following constant:

\begin{defn}
\[C(\rho_0, \rho) := \inf\{\textrm{Lip}(f):  f \mbox{ Lipschitz and } (\rho_0, \rho)\mbox{-equivariant}\}\]
where $\textrm{Lip}(f)$ is the Lipschitz constant of $f$.
\end{defn}

\noindent They then prove:

\begin{thm}[Gu\'eritaud--Kassel; see Chapter 5 of \cite{fanny_thesis} for the case $n=2$ and \cite{g-k1} Theorem 1.8 for the full result]\label{gkthm}
The pair $(\rho_0, \rho)$, with $\rho$ geometrically finite, is admissible if and only if, up to switching $\rho_0$ and $\rho$, the representation $\rho_0$ is injective and discrete and $C(\rho_0, \rho) < 1$.
\end{thm}

This result comes out of a larger project of understanding the import of the constant $C(\rho_0, \rho)$.  In particular, they study $f$ achieving $C(\rho_0, \rho)$ as their Lipschitz constant, and the geometry of the \emph{stretch locus} for such $f$, i.e. those points in $\mathbb{H}^n$ which have no neighborhood over which the Lipschitz constant for $f$ is smaller than the global constant $\textrm{Lip}(f)$. The `if' part of Theorem \ref{gkthm} is a consequence of Lemma \ref{salein_criterion} with proof adapted to the case of $\mathbb{H}^n$; the `only if' follows from careful understanding of the structure of the stretch locus when $C(\rho_0, \rho)\geq 1$.

\vspace{.5cm}

In \cite{kassel_deformation}, Kassel improves on Kobayashi's Proposition \ref{kob_def_result}:

\begin{thm}[Kassel, \cite{kassel_deformation} Thm 1.1]\label{kassel_def}
Let $G$ be a real, reductive, linear Lie group; let $H$ and $L$ be closed, reductive subgroups with $\mathbb{R}$-$\emph{rank}(L)=1$ with $L$ acting properly discontinuously and cocompactly on $H\backslash G$.  For any uniform lattice $\Gamma$ of $L$, there exists a neighborhood $\mathcal{U}$ of the natural inclusion in $\emph{Hom}(\Gamma, G)$ such that any $\phi\in \mathcal{U}$ satisfies:
\begin{itemize}
	\item $\phi(\Gamma)$ is discrete in $G$,
	\item $\phi(\Gamma)$ acts properly discontinuously and cocompactly on $H\backslash G$.
\end{itemize}
\end{thm}

As a corollary, Kassel obtains

\begin{cor}[Kassel, \cite{kassel_deformation} Corollary 1.2]\label{kas_cor_1}
There are Zariski-dense $\Gamma$ in $\emph{SO}(2,2n)$ providing compact forms of $\emph{U}(1,n)\backslash \emph{SO}(2,2n)$.
\end{cor}

\noindent Prior to Kassel's work, the only Zariski-dense $\Gamma$ known (for $H$ noncompact) were for homogeneous spaces of the form $diag(G')\backslash(G'\times G')$. The $\mathbb{R}$-$\textrm{rank} = 1$ condition is natural as Margulis superrigidity implies that uniform lattices $\Gamma < L$ of higher rank are locally rigid in $G$.

The idea of Kassel's proof is to study the Cartan projection for elements of $\Gamma$.  She then uses some interesting dynamics of the action of $G$ on $\mathbb{P}(V)$ for representations of $G$ on $V$.  The study of these dynamics allows one to decompose any $\gamma \in \Gamma$ as a product of elements from a finite set whose Cartan projections can be carefully controlled.

%
%%
%%%
%%%%
%%%%%
%%%%%%
%%%%%%%%%%%%%%%%%%%%%%%%%%%%%%%%%%%%
\subsubsection{The work of Guichard and Wienhard}

Recent work by Guichard and Wienhard on Anosov representations has yielded some corollaries about deformations of compact forms.  In \cite{g-w_anosov}, they prove the following:

\begin{thm}[Guichard--Wienhard, \cite{g-w_anosov} \S13]
Let $L=\emph{SO}(1,2n)$, embedded in $G=\emph{SO}(2,2n)$ in the standard way.  Let $\Gamma$ be a uniform lattice in $\emph{SO}(1,2n)$ and $\rho$ the induced embedding of $\Gamma$ into $\emph{SO}(2,2n)$.  Then any representation in the connected component of $\rho$ in $\emph{Hom}(\Gamma, \emph{SO}(2,2n))$ yields a compact form of $\emph{U}(1,n)\backslash\emph{SO}(2,2n)$.
\end{thm}

That non-trivial deformations of compact forms for these spaces exist is not new (see Prop. \ref{kob_def_result} and Thm. \ref{kassel_def}).  However, Guichard and Wienhard produce them in a new way.  Their paper deals with Anosov representations of discrete groups into semisimple Lie groups $G'$ -- a definition due to Labourie (\cite{lab_anosov}) that they prove unifies many previous examples of special representations.  To each such representation they can associate a flag variety of $G'$ (formed by quotienting out by a parabolic subgroup) and on this flag variety there is a domain of discontinuity on which $\Gamma$ acts properly discontinuously and cocompactly.  In some special cases, this domain has the structure of a homogeneous space for some $G < G'$, giving rise to their examples.  For the theorem above, they use Barbot's result (\cite{barbot}) that the full connected component of $\rho$ consists of Anosov representations (relative to a certain parabolic subgroup) to obtain a stronger conclusion than Theorem \ref{kassel_def} and Corollary \ref{kas_cor_1} have obtained: \emph{all} of these representations give compact forms.

These examples are a small part of an extensive paper.  Guichard and Wienhard do not undertake an exhaustive search for applications to compact forms; it is likely that there are more to be found.

%In some very recent work, Gu\'eritaud and Kassel have extended the work above.  In \cite{g-k2}, they study deformations of compact forms for rank-one groups.  Let $G$ be a semisimple group with real-rank one and $\Gamma$ a discrete subgroup of $G\times G$ acting freely, properly discontinuously and cocompactly on $G$ by left- and right-multiplication.  They prove that, once again, these compact forms can be locally deformed:

%\begin{thm}[Gu\'eritaud-Kassel, \cite{g-k2} Thm 1]
%In the setting above, there exists an open neighborhood $\mathcal{U}\subset Hom(\Gamma, G\times G)$ of the natural inclusion such that for all $\phi \in \mathcal{U}$, $\phi(\Gamma)$ is discrete in $G\times G$ and still acts freely, properly discontinuously and cocompactly on $G$, yielding a compact form.
%\end{thm}

%\noindent Gu\'eritaud and Kassel actually prove that a slightly larger class of $\Gamma$ -- those they call  \emph{convex cocompact} can be deformed.

%
%%
%%%
%%%%
%%%%%
%%%%%%%%%%%%%%%%%%%%%%%%%%

\subsection{Deformations of the homogeneous space}

To close this section, I would like to briefly touch on the work of Oh, Witte-Morris and Iozzi on deformations of compact forms in another sense.  Rather than deforming the embedding of $\Gamma$ in $G$ to produce new compact forms of $H\backslash G$, they deform the embedding of $H$ in $G$, thereby producing new homogeneous spaces that admit nonstandard compact forms.  In this approach the analogy with Teichm\"uller theory breaks down.  We are no longer studying the deformation space of compact forms of a given homogeneous space; rather, the underlying homogeneous space changes.  These results, rather, speak to the wide (and likely wild) world of compact forms that exist.

Oh and Witte-Morris's work is announced in the paper \cite{Oh_Witte_new} and treated in full detail in \cite{Oh_Witte_compact}.  Iozzi and Witte-Morris continue this work in \cite{Witte_Iozzi}. Oh and Witte-Morris deal with homogeneous spaces of $\textrm{SO}(2,n)$ and their main result divides into two pieces depending on the parity of $n$:

Let $G=\textrm{SO}(2, 2m)$, presented as the group preserving the form $2x_1x_{n+2}+2x_2x_{n+1}+\sum_{i=3}^{n}x_i^2$.  Let $H_{\textrm{SU}}$ be the intersection of $\textrm{SU}(1, m)$ (embedded in $G$ in the standard way) with $AN$ where $A$ consists of the diagonal elements in $G$ with positive entries and $N$ consists of the upper-triangular matrices with ones along the diagonal.  Note that $SU(1,m)/H_{\textrm{SU}}$ is compact and that, in contrast to the reductive situation discussed through most of this survey, $H_{\textrm{SU}}$ is solvable. Let $\Gamma$ be a uniform lattice in $\textrm{SO}(1,2m)$; it will act properly discontinuously and cocompactly on $H_{\textrm{SU}}\backslash \textrm{SO}(2,2m)$ because it does so on $\textrm{SU}(1,m)\backslash \textrm{SO}(2,2m)$.

\begin{thm}[Oh--Witte-Morris \cite{Oh_Witte_compact} Thms 1.5 \& 1.7]
\begin{itemize}
	\item[(1)] For a specific family of deformations $H_B$ of $H_{\emph{SU}}$, $H_B\backslash G/\Gamma$ yields a compact form.  Moreover, for almost all choices of the deforming parameter $B$, the group $H_B$ is not conjugate into $\emph{SU}(1,m)$, so these are truly new examples. (See their paper for details on $H_B$.)
	\item[(2)] Among closed, connected, upper triangular, noncompact subgroups $H$ of the Borel subgroup of $G$ such that $H \backslash G$ is noncompact, the only $H\backslash G$ possessing compact forms are those which are conjugate to a cocompact subgroup of $\emph{SO}(1, 2m)$ or to one of the $H_B$.  (See section 3.9 of their paper for an argument that reduces any closed, connected $H$ to the upper triangular case.)
\end{itemize}
\end{thm}

\noindent A streamlined proof of this result is provided in \cite{Witte_Iozzi}; the key advance is an \emph{a priori} lower bound on the dimension of $H$, which makes the subsequent case-by-case analysis quicker.  For specifics on $H_B$ we refer the reader to \cite{Oh_Witte_compact}, but remark that the constructions are entirely explicit.  The upshot of this theorem is that the homogeneous spaces of $\textrm{SO}(2,2m)$ are entirely understood.  Those $H$ conjugate into $\textrm{SO}(1, 2m)$ return us to the examples studied by Kobayashi in \cite{kobayashi_def} with their nontrivial deformation space.

Now let $G=\textrm{SO}(2, 2m+1)$.  As the algebraic construction of Conjecture \ref{kob_conj} does not hold for this $G$ and $H=\textrm{SU}(1, m)$ it is conjectured that there are no compact forms of $\textrm{SU}(1,m)\backslash G$.  This problem is still open, but Oh and Witte-Morris show that its solution will settle all questions for homogeneous spaces of $\textrm{SO}(2, 2m+1)$:

\begin{thm}[Oh--Witte-Morris \cite{Oh_Witte_compact} Thm 1.9]
Let $G=\emph{SO}(2, 2m+1)$ and let $H$ be closed, connected, noncompact with $H\backslash G$ noncompact.  If $\emph{SU}(1,m)\backslash G$ has no compact form, then neither does $H\backslash G$.
\end{thm}

Finally, Iozzi and Witte-Morris obtain the following analogous result for homogeneous spaces of $G=\emph{SU}(2, 2m)$:

\begin{thm}[Iozzi--Witte-Morris, \cite{Witte_Iozzi} Thm 11.5$''$]
Let $G=\emph{SU}(2,2m)$ and let $H$ be a closed, connected, noncompact subgroup of $G$ with $H\backslash G$ noncompact.  Then $H\backslash G$ admits a compact Clifford--Klein form if and only if $d(H)=4m$ and $H$ belongs to a specific family of deformations of $\emph{SU}(1,2m)$ or $\emph{Sp}(1,m)$ in $G$. (Recall that $d(H)$ was defined in Prop \ref{signature_prop}.)
\end{thm}

\noindent Again, the reader is referred to the paper for specifics of the deformations, but they are explicitly given.

%
%%
%%%
%%%%
%%%%%%%%%%%%%%%%%%%%%%%%%%

\section{The road ahead}

Despite the extensive work on the existence question for compact forms, there are still many open cases.  The author's favorite is $\SL {n-2}\backslash \SL {n}$.\footnote{Recall that the case $\SL 2 \backslash \SL 4$ is known to have no compact form by Corollary \ref{shalom_cor}.}  Two very different approaches to the problem -- the `topological' approach via Cartan projections of Benoist and Kobayashi, and the dynamical approach initiated by Zimmer -- come right up to this problem, but both fail critically.  We certainly expect there are no compact forms, but a new idea seems necessary, even for such an algebraically simple example.  There is no shortage of other homogeneous spaces for which the existence question is open.

The deformation question is somewhat less developed, and there is certainly plenty more to be done.  The work of Salein indicates how interesting a moduli space for even the simplest examples will be, and one expects that there is far more to the full moduli space than he and Gu\'eritaud--Kassel have discovered.  The work of Oh, Witte-Morris and Iozzi indicates that there are plenty of new compact forms to be discovered as we loosen the restrictions on the homogeneous space.

I would like to close this survey, however, by briefly mentioning a very recent result of Kassel and Kobayashi in which they have gone beyond the existence and deformation questions, and begun to study the spectral theory of these spaces.

Let $\sigma$ be an involutive automorphism of $G$ and let $H=(G^\sigma)_o$ be the identity component of the set of fixed points, so that $H\backslash G$ is a pseudo-Riemannian symmetric space.  Let $\theta$ be a Cartan involution of $G$ commuting with $\sigma$ and $K$ the corresponding maximal compact subgroup.  Let $\textrm{Spec}_d(\dq)$ be the set of eigenvalues associated to $L^2$ eigenfunctions of the pseudo-Riemannian Laplacian, i.e. the discrete spectrum of this operator.  Kobayashi and Kassel introduce the notion of `sharpness' in their paper \cite{ko-kas}.  Roughly speaking, we say that the pair $(H, \Gamma)$ satisfies the \emph{sharpness condition} if the Cartan projection of $\Gamma$ diverges \emph{linearly} from the Cartan projection of $H$ as one heads to infinity in the Cartan subgroup.  A precise formulation can be found in \cite{ko-kas} \S 1.6.

\begin{thm}[Kassel--Kobayashi, announced in \cite{ka-ko}, detailed proofs in \cite{ko-kas}]
Suppose that $\emph{rank}(H\backslash G) = \emph{rank} ((K\cap H)\backslash H)$ where by rank we mean the dimension of a maximal, semisimple, abelian subspace in the set of fixed points of $-d\sigma$.  Suppose that the pair $(H, \Gamma)$ satisfies the sharpness condition. Then: 
\begin{enumerate}
	\item For any compact form $\dq$, the spectrum $\emph{Spec}_d(\dq)$ is infinite.  
	\item For standard compact forms with $\Gamma < L$ and $\mathbb{R}$-rank$(L)=1$ and for all compact forms of $diag(\emph{SO}(1,n))\backslash (\emph{SO}(1,n)\times \emph{SO}(1,n))$, there is an infinite subset of $\emph{Spec}_d(\dq)$ which is stable under any small deformation of $\dq$.
\end{enumerate}
\end{thm}

\noindent  For standard forms, the sharpness condition is always satisfied.  In certain other cases -- for example $\textrm{AdS}^3$-manifolds -- it is known that all compact forms are sharp \cite{fanny_thesis}.  Kassel and Kobayashi conjecture that it is always satisfied.

This nice result gives us an indication of one road ahead for the study of Clifford--Klein forms.  There are many basic, unanswered geometric questions about these spaces; they are likely to be much more difficult in the pseudo-Riemannan case than in the Riemannian case for the reasons we have noted above.  Very little work has been done in this direction, but the results surveyed here provide a very wide variety of tools to address such problems, as well as a library of examples on which to test them.

\bibliographystyle{alpha}
\bibliography{biblio}

\end{document}